\documentclass[conference, a4paper]{IEEEtran}
\IEEEoverridecommandlockouts
\usepackage{amsthm}
\usepackage{cite}
\usepackage{amsmath, amssymb, amscd, latexsym, bm, mathtools, braket, multirow, url}
\usepackage{balance}

\newcounter{count}
\newcommand{\num}{\stepcounter{count}\the\value{count}}

\newtheorem{definition}{Definition}
\newtheorem{theorem}{Theorem}

\newtheorem{lemma}{Lemma}
\newtheorem{corollary}{Corollary}

\theoremstyle{definition}
\newtheorem*{remark*}{Remark}

\DeclareMathOperator*{\argmax}{arg~max}

\begin{document}
\title{
Optimal Mechanism for Randomized Responses under Universally Composable Security Measure
\thanks{
MH was supported in part by 
Kayamori Foundation of Informational Science Advancement
and a JSPS Grant-in-Aid for Scientific Research 
(A) No.17H01280, (B) No.16KT0017. 
MHY was supported by National Natural ScienceFoundation of China (Grants No.~11405093), 
Natural Science Foundation of Guangdong Province (2017B030308003), 
the Guangdong Innovative and Entrepreneurial Research Team Program (Grant No.~2016ZT06D348), 
and the Science Technology and Innovation Commission of Shenzhen Municipality (Grants No.~ZDSYS20170303165926217 and No.~JCYJ20170412152620376).
}
}

\author{
\IEEEauthorblockN{Yuuya Yoshida$^{a}$, Man-Hong Yung$^{b,c,d}$, and Masahito Hayashi$^{a,b}$}\\
\IEEEauthorblockA{
$^a$Graduate School of Mathematics, Nagoya University\\
$^b$Shenzhen Institute for Quantum Science and Engineering, Southern University of Science and Technology\\
$^c$Department of Physics, Southern University of Science and Technology\\
$^d$Shenzhen Key Laboratory of Quantum Science and Engineering\\
Email: {m17043e@math.nagoya-u.ac.jp,  yung@sustc.edu.cn,  masahito@math.nagoya-u.ac.jp}
}
}

\maketitle
\begin{abstract}
We consider a problem of analyzing a global property of private data 
through randomized responses subject to a certain rule,
where private data are used for another cryptographic protocol, e.g., authentication.
For this problem, the security of private data was evaluated by a universally composable security measure, 
which can be regarded as $(0,\delta)$-differential privacy.
Here we focus on the trade-off between the global accuracy and a universally composable security measure,
and derive an optimal solution to the trade-off problem. 
More precisely, we adopt the Fisher information of a certain distribution family as the estimation accuracy of a global property 
and impose $(0,\delta)$-differential privacy on a randomization mechanism protecting private data. 
Finally, we maximize the Fisher information under the $(0,\delta)$-differential privacy constraint 
and obtain an optimal mechanism explicitly.
\end{abstract}

\begin{IEEEkeywords}
universally composable security measure,
$(0,\delta)$-differential privacy,
$l^1$-norm,
Fisher information,
parameter estimation,
sublinear function
\end{IEEEkeywords}

\section{Introduction}\label{S1}
For many applications, it is of great interest in estimating a global property of an ensemble while protecting individual privacy. 
In these scenarios, disclosed data are randomized to protect individual privacy, but it causes uncertainty to the global property. 
This fact implies a trade-off between global accuracy and individual privacy. 
This paper focuses on a scenario of answering YES/NO question, 
where the goal is to estimate the number of YES's (labeled as ``1'') or NO's (labeled as ``0'') with high accuracy, 
given that respondents randomize their responses to protect individual privacy.

More specifically, we assume that an investigator is interested only in the ratio $1-\theta:\theta$ of the binary private data, 
where $\theta$ is a real number between zero and one. 
The investigator randomly chooses $n$ individuals to ask for their private data. 
If selected individuals directly sent their private data $X\in\mathcal{X}:=\{0,1\}$ to the investigator, 
then their private data could be completely leaked to the investigator. 
To protect individual privacy, we use the following scheme \cite{Warner,Greenberg}: 
when private data is $X=0$, the individual generates a disclosed data $Y \in \mathcal{Y}$ 
subject to a distribution $p_0$ on a probability space $\mathcal{Y}$ and sends it to the investigator.
In the following, the sets $\mathcal{X}$ and $\mathcal{Y}$ are called the private data set and the disclosed data set, 
and $|\mathcal{S}|$ denotes the number of elements of a set $\mathcal{S}$.
When private data is $X=1$, the individual generates a disclosed data $Y$ 
subject to another distribution $p_1$ on $\mathcal{Y}$ and sends it to the investigator.

To analyze estimation of the ratio $1-\theta : \theta$, 
we introduce the parametrized distribution $p_\theta$ defined as 
$p_\theta = (1-\theta)p_0+ \theta p_1$.
In this way, the estimation of the ratio $1-\theta : \theta$ is reduced to 
the estimation of the parameter $\theta \in [0,1]$ of the distribution family $\{p_\theta\}_{ \theta \in [0,1]}$ 
when $n$ data are generated from the same unknown distribution $p_\theta$ independently, 
as we allow duplication in the selection of individuals. 
In literatures \cite{Cam1,Cam2,Vaart} of statistical parameter estimation, 
it is well-known that an optimal estimator is given as the maximum likelihood estimator (MLE) 
with respect to sufficiently large $n$ data.
The error of the MLE is asymptotically characterized by the inverse of the Fisher information $J_\theta$, 
which is called the Cram\'{e}r-Rao bound. 
Thus we adopt the Fisher information $J_\theta$ as the estimation accuracy.

In addition to the above scenario, 
it is natural to use private data as resources for another cryptographic protocol like an authentication protocol \cite{KR94}. 
In fact, we often use private data, e.g., birthday, to identify an individual.
In this case, we need to guarantee the security of the whole protocol.
That is, if a part of information of private data is leaked,
we need to consider its effect to the cryptographic protocol that uses private data. 
In the cryptography community, to evaluate the security of the whole protocol, 
a security measure based on the $l^1$-norm is proposed as a \textit{universally composable (UC)} security measure \cite{Cannett}.
If the UC-security measure of the first protocol equals $\delta$ 
and that of the second protocol equals $\delta'$,
then that of the combined protocol is upper bounded by $\delta+\delta'$. 
This property is called \textit{universal composability}. 
Thanks to this property, the $l^1$-norm is widely accepted as a security measure
in the communities of cryptography and information-theoretic security \cite{MW,RW,Hayashi}.
Thus, when using private data for a cryptographic protocol,
we need to guarantee that the UC-security measure is upper bonded by a certain threshold. 
On the other hand, as a privacy measure, Kairouz et al.\ \cite{KOV16} focused on $(\epsilon,0)$-differential privacy 
and maximized the Fisher information $J_\theta$ under the $(\epsilon,0)$-differential privacy constraint.
\textit{Differential privacy (DP)} is a standard privacy measure 
that is widely accepted and introduced by \cite{DMNS} and \cite{Dwork}. 
(For the definition, see \eqref{DP} in Section~\ref{S2}.)
However, $(\epsilon,0)$-differential privacy does not have the universally composable property 
for private data.
Therefore, we need to address the trade-off between the Fisher information and the UC-security measure.

In our setting, the UC-security measure is given as the variational distance 
$d_1(p_0,p_1):=(1/2)\|p_0-p_1\|_1$ between two distributions $p_0$ and $p_1$,
where $\|\cdot\|_1$ denotes the $l^1$-norm. 
The variational distance also has the following meaning: 
if an adversary tries to distinguish private date, the minimum value of the average error probability equals 
$\min_{S\subset\mathcal{Y}}\{ (1/2)p_0(S) + (1/2)p_1(S^c) \} = (1/2)(1-d_1(p_0,p_1))$, 
where $S^c$ denotes the complement of $S$. 
Fortunately, it can be regarded as $(0,\delta)$-differential privacy. 
Thus we can also say that 
we maximize the Fisher information $J_\theta$ under the $(0,\delta)$-differential privacy constraint.

Further, to address the above maximization, we encounter a new aspect 
that never appeared in preceding studies 
for the trade-off between the Fisher information and $(\epsilon,0)$-differential privacy.
Kairouz et al.\ \cite{KOV16} maximized (non-explicitly) the Fisher information under the $(\epsilon,0)$-differential privacy constraint 
when $|\mathcal{X}|,|\mathcal{Y}|\ge2$.
Then they showed that the maximization under the $(\epsilon,0)$-differential privacy constraint 
achieves the maximum value even when $|\mathcal{X}|=|\mathcal{Y}|$.
Holohan et al.\ \cite{HLM} considered the maximization under the $(\epsilon,\delta)$-differential privacy 
constraint\footnote{Although $(\epsilon,\delta)$-differential privacy prevents \textit{blatant non-privacy} \cite{De}, 
blatant non-privacy is related to the privacy of data sequences. 
This relation is out of our focus because our main interest is individual privacy.}. 
However, they assumed $|\mathcal{X}|=|\mathcal{Y}|=2$. 
Therefore, this kind of maximization has been open 
for a general disclosed data set $\mathcal{Y}$ even if $|\mathcal{X}|=2$. 
To find an optimal mechanism in our framework,
we need to maximize the Fisher information for a general disclosed data set $\mathcal{Y}$.
In fact, we can show that the maximization under the $(0,\delta)$-differential privacy constraint 
achieves the maximum value only when $|\mathcal{Y}|>|\mathcal{X}|$. 
As a result, we obtain a randomized response scheme with 
$|\mathcal{X}|=2$ and $|\mathcal{Y}|=3$
that has been never obtained in preceding studies.
Our optimal solution is completely different from those of \cite{KOV16} and \cite{HLM}.
To handle the case with $|\mathcal{Y}|>|\mathcal{X}|$,
we need to address an additional case that is more 
complicated than the case with $|\mathcal{X}|=|\mathcal{Y}|$.
Table~\ref{T2} summarizes the relation among \cite{KOV16}, \cite{HLM} and this paper.

\begin{table}[t]
\caption{Relation with preceding studies}\label{T2}
This table shows 
differential privacy constraints and 
conditions for private data sets $\mathcal{X}$ and disclosed data sets $\mathcal{Y}$.
\begin{center}
\begin{tabular}{|c||c|c|c|c|}
	\hline
	   &\multirow{2}{*}{Constraint}&\multirow{2}{*}{$|\mathcal{X}|$}&\multirow{2}{*}{$|\mathcal{Y}|$}&Condition\\
	   &   &   &   &for optimality\\
	\hline\hline
	Kairouz et al.\ \cite{KOV16}&$(\epsilon,0)$-DP&$\ge2$&$\ge2$&$|\mathcal{X}|=|\mathcal{Y}|$\\
	\hline
	Holohan et al.\ \cite{HLM}&$(\epsilon,\delta)$-DP&$2$&$2$&   \\
	\hline
	This paper&$(0,\delta)$-DP&$2$&$\ge2$&$|\mathcal{Y}|\ge3$\\
	\hline
\end{tabular}
\end{center}
\end{table}

Since the $l^1$-norm represents the minimum value of the average error probability in distinguishing private data,
it is natural to consider the minimum value of the weighted error probability 
$\min_{S\subset \mathcal{Y}}\{ (1-w)p_0(S)+w p_1(S^c) \}$ with an arbitrary weight $w\in(0,1)$.
In this case, it equals $(1/2)(1- \| (1-w)p_0-w p_1 \|_1)$.
Hence we consider the extended trade-off including $\| (1-w)p_0-w p_1 \|_1$ instead of $d_1(p_0,p_1)$. 
Furthermore, this paper also addresses the following scenario: 
when the true parameter is known to be $\theta_1$ or $\theta_2$,
our problem reduces to the discrimination between two distributions $p_{\theta_1}$ and $p_{\theta_2}$.
When we use an optimal testing method, 
the error probability goes to zero exponentially as the number of observations $n$ tends to infinity.
Dependently on our setting,
the optimal exponentially decreasing rate is known as the Chernoff bound or Stein's lemma \cite{Chernoff,Cover}.
Hence we optimize these exponents under the same constraint.
To do the above optimizations, 
we optimize a more general objective function, which is the sum of values of a \textit{sublinear function}, in Section~\ref{S3}. 
Although this objective function was optimized in \cite{KOV16} under the $(\epsilon,0)$-differential privacy constraint, 
we optimize it under the $(0,\delta)$-differential privacy constraint. 

The remaining part of this paper is organized as follows. 
Section~\ref{S2} states the formulation of our problem and the maximum Fisher information under the $l^1$-norm constraint, 
which describes that the maximum Fisher information depends on the weight $w$ and the parameter $\theta$.
Section~\ref{S3} proves our theorem on the maximum Fisher information
by solving an optimization with a general sublinear function.
Section~\ref{S4} shows that any optimal pair of two distributions $p_0$ and $p_1$ depends on the weight $w$.
By applying a general result in Section~\ref{S3},
Section~\ref{S5} discusses the maximization of the exponentially decreasing rate of the error probability.
Section~\ref{S6} explains the relation among preceding studies and this paper.
Section~\ref{S7} devotes concluding remarks.

\section{Optimal estimation}\label{S2}
According to Fig.~\ref{F1}, we describe a scheme to estimate the parameter $\theta$.
In our scheme, we first fix two distributions $p_0$ and $p_1$ on a finite probability space $\mathcal{Y}$,
which describes a conversion rule from private data $X_i\in\mathcal{X}=\{0,1\}$ to disclosed data $Y_i\in\mathcal{Y}$.  
Assume that private data $X_1, \ldots, X_n \in \mathcal{X}$ are independent and 
subject to the binary distribution parametrized by the parameter $\theta \in (0,1)$. 
That is, the true probability of $X_i=0$ is $1-\theta$ and that of $X_i=1$ is $\theta$,
where $X_i$ denotes the $i$-th individual's private data.
The $i$-th individual generates a disclosed data $Y_i$ subject to $p_x$ dependent on the value $X_i=x$ 
and then sends $Y_i$ to the investigator.
From the investigator's viewpoint, 
the disclosed data $Y_i$ is given by the distribution $p_\theta := (1-\theta)p_0+ \theta p_1$.

Next, the parameter $\theta$ is estimated by the investigator in the following way. 
The investigator observes $n$ disclosed data $Y_1, \ldots, Y_n$ 
and employs the MLE 
$\hat{\theta}_n:= \argmax_{\theta\in[0,1]} \sum_{i=1}^n \ln p_\theta(Y_i)$, 
whose asymptotic optimality is well-known \cite{Cam1,Cam2,Vaart}. 
That is, if $n$ is sufficiently large, 
the MLE $\hat{\theta}_n$ behaves approximately as 
$\theta+ (n J_\theta)^{-1/2} Z$, where
$Z$ is a random variable subject to the standard Gaussian distribution and 
$J_\theta$ is the Fisher information of the distribution family $\{p_\theta\}_{\theta\in[0,1]}$: 
\[
J_\theta = \sum_{y\in\mathcal{Y}} \Bigl( \frac{d}{d\theta}\ln p_\theta(y) \Bigr)^2 p_\theta(y).
\]
Therefore, the mean square error behaves as $1/n J_\theta$.
For example, when we require the confidence level to be $\alpha$,
the confidence interval is approximately given as 
\[
\Bigl[ \hat{\theta}_n + \frac{1}{\sqrt{n J_\theta}} \Phi^{-1}(\alpha/2),\ 
\hat{\theta}_n + \frac{1}{\sqrt{n J_\theta}} \Phi^{-1}(1 - \alpha/2) \Bigr],
\]
where
$\Phi(y) := (2\pi)^{-1/2}\int_{-\infty}^y e^{-x^2/2}\,dx$.
In this sense, we can conclude that the Fisher information $J_\theta$ is the estimation accuracy.

\begin{figure}[t]
	\centering
	\includegraphics[scale=0.45]{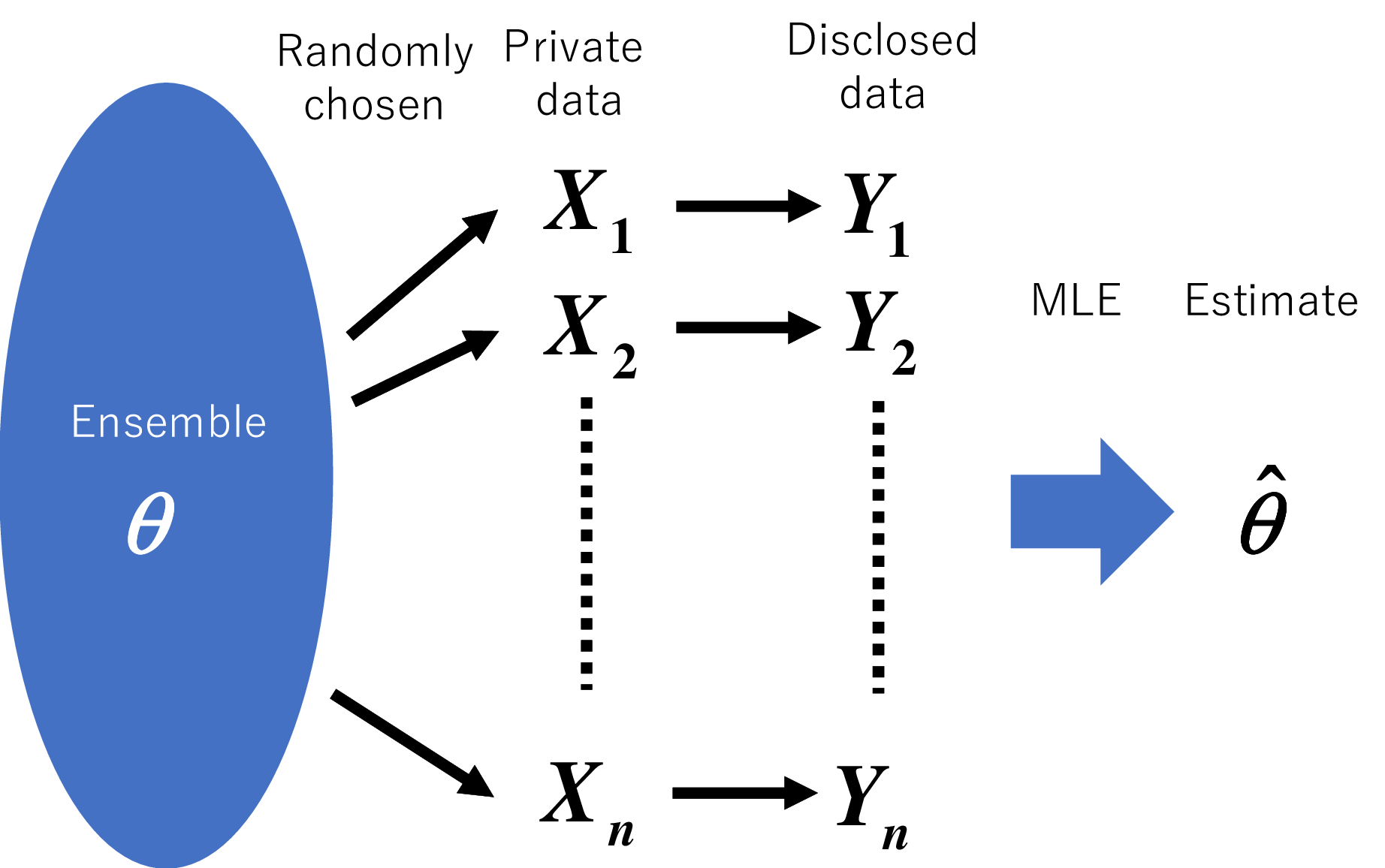}
	\caption{Estimation process of parameter $\theta$.}
	\label{F1}
\end{figure}

If the investigator only observes the $i$-th disclosed data $Y_i$, 
the investigator can also infer, at least partially, the $i$-th private data $X_i$.
Since we assume to use the private data $X_i$ in another cryptographic protocol,
the UC-security measure $\| p_0-p_1 \|_1$ is suitable for a privacy criterion.
The UC-security measure characterizes the distinguishability as follows.
The minimum value of the weighted error probability of the above inference with a weight $w \in (0,1)$ is 
$\min_{S\subset\mathcal{Y}}\{ (1-w)p_0(S)+w p_1(S^c) \}$,
where $S$ is a subset of events to infer $X_i=1$.
When treating two error probabilities $p_0(S)$ and $p_1(S^c)$ equally, 
the weight $w$ is one half.
When taking the error probability $p_0(S)$ more seriously than $p_1(S^c)$,
the weight $w$ is smaller than one half.
Now, the minimum value of the weighted error probability is calculated as
\begin{equation*}
	\min_{S\subset \mathcal{Y}}\{ (1-w)p_0(S)+w p_1(S^c) \}
	= \frac{1}{2}(1- \| (1-w)p_0-w p_1 \|_1),
\end{equation*}
which quantifies the distinguishability. 
Thus, to keep individual privacy a given level, 
we impose a constant constraint on the $l^1$-norm 
as $\| (1-w)p_0-w p_1 \|_1\le\delta$,
which is equivalent to the condition 
\begin{equation}
	\min_{S\subset \mathcal{Y}}\{ (1-w)p_0(S)+w p_1(S^c) \} \ge a := \frac{1-\delta}{2}. \label{KKO}
\end{equation}

To protect individual privacy, 
many preceding studies consider \textit{$(\epsilon,\delta)$-differential privacy}, 
which is defined by the following condition: for any $S\subset\mathcal{Y}$ and any $i,j\in\mathcal{X}$, 
\begin{equation}
	p_i(S) \le e^\epsilon p_j(S) + \delta. \label{DP}
\end{equation}
When $\epsilon=0$, the above condition can be simplified to $(1/2)\|p_0-p_1\|\le\delta$. 
Thus we can also say that the $l^1$-norm constraint is $(0,\delta)$-differential privacy when $w=1/2$.
As explained in the previous paragraph, we also address the minimum value of the weighted error probability as a privacy criterion.
Hence we maximize the Fisher information $J_\theta$ under the $l^1$-norm constraint $\|(1-w)p_0 - w p_1\|_1\le\delta$.
In the following, we denote the Fisher information by $J_\theta(p_0,p_1)$.

From now on, we shall discuss only the case $\delta\in(0,1)$ because of the following reasons: 
the triangle inequality yields $\|(1-w)p_0 - w p_1\|_1 \le (1-w) + w = 1$, 
which implies the inequality $\|(1-w)p_0 - w p_1\|_1 \le \delta$ whenever $\delta\ge1$; 
the condition $\delta\ge1$ implies that 
the domain of the maximization problem is the set of all pairs of two distributions, 
which means that there is no constraint; 
the condition $\delta=0$ implies that 
the domain is either the empty set or the set of all pairs $(p_0,p_1)$ of two distributions with $p_0=p_1$, 
which means that the Fisher information $J_\theta(p_0,p_1)$ vanishes identically. 
Further, since the triangle inequality yields $|1-2w| = |(1-w) - w| \le \| (1-w)p_0-w p_1 \|_1$,
the constraint $\| (1-w)p_0-w p_1 \|_1 \le \delta$
implies $|1-2w| \le \delta$, which is equivalent to 
the condition $a = (1-\delta)/2 \le w \le (1+\delta)/2 = 1-a$.
Hence we also impose $w\in[a,1-a]$ on $w\in(0,1)$. 
As the trade-off between the estimation accuracy $J_\theta(p_0,p_1)$ and the individual privacy $\| (1-w)p_0-w p_1 \|_1$, 
we have the following theorem:

\begin{theorem}\label{c-var}
	Assume the conditions $w\in[a,1-a]$ and $\delta,\theta\in(0,1)$, 
	and set the parameters $a := (1-\delta)/2$ and $\theta_0:= (w-a)/\delta$.
	Then we have three cases: 
	(i) $|\mathcal{Y}|=2$ and $\theta\le\theta_0$,
	(ii) $|\mathcal{Y}|=2$ and $\theta>\theta_0$,
	and (iii) $|\mathcal{Y}|\ge3$.
	In these three cases, from top to bottom, we have 
	\begin{align*}
		&\max_{\|(1-w)p_0 - w p_1\|_1\le\delta} J_\theta(p_0,p_1)\\
		=&
		\begin{dcases}
			\frac{w-a}{\theta\{ w(1-\theta) + a\theta \}},\\
			\frac{1-w-a}{(1-\theta)\{ a(1-\theta) + (1-w)\theta \}},\\
			\frac{1}{\theta(1-\theta)}\Bigl( 1 - \frac{a}{w(1-\theta) + (1-w)\theta} \Bigr).
		\end{dcases}
	\end{align*}
	In these three cases, from top to bottom,
	distributions $p_0$ and $p_1$ to achieve the maximum value 
	are given as 
	\begin{equation*}
		p_0 = 
		\begin{dcases}
			[1,0],\\
			\Bigl[ \frac{a}{1-w}, 1-\frac{a}{1-w} \Bigr],\\
			\Bigl[ \frac{a}{1-w}, 1-\frac{a}{1-w}, 0 \Bigr],
		\end{dcases}
		\quad p_1 = 
		\begin{dcases}
			\Bigl[ \frac{a}{w}, 1-\frac{a}{w} \Bigr],\\
			[1,0],\\
			\Bigl[ \frac{a}{w}, 0, 1-\frac{a}{w} \Bigr].
		\end{dcases}
	\end{equation*}
\end{theorem}

\begin{figure}[t]
	\centering
	\includegraphics[scale=0.5]{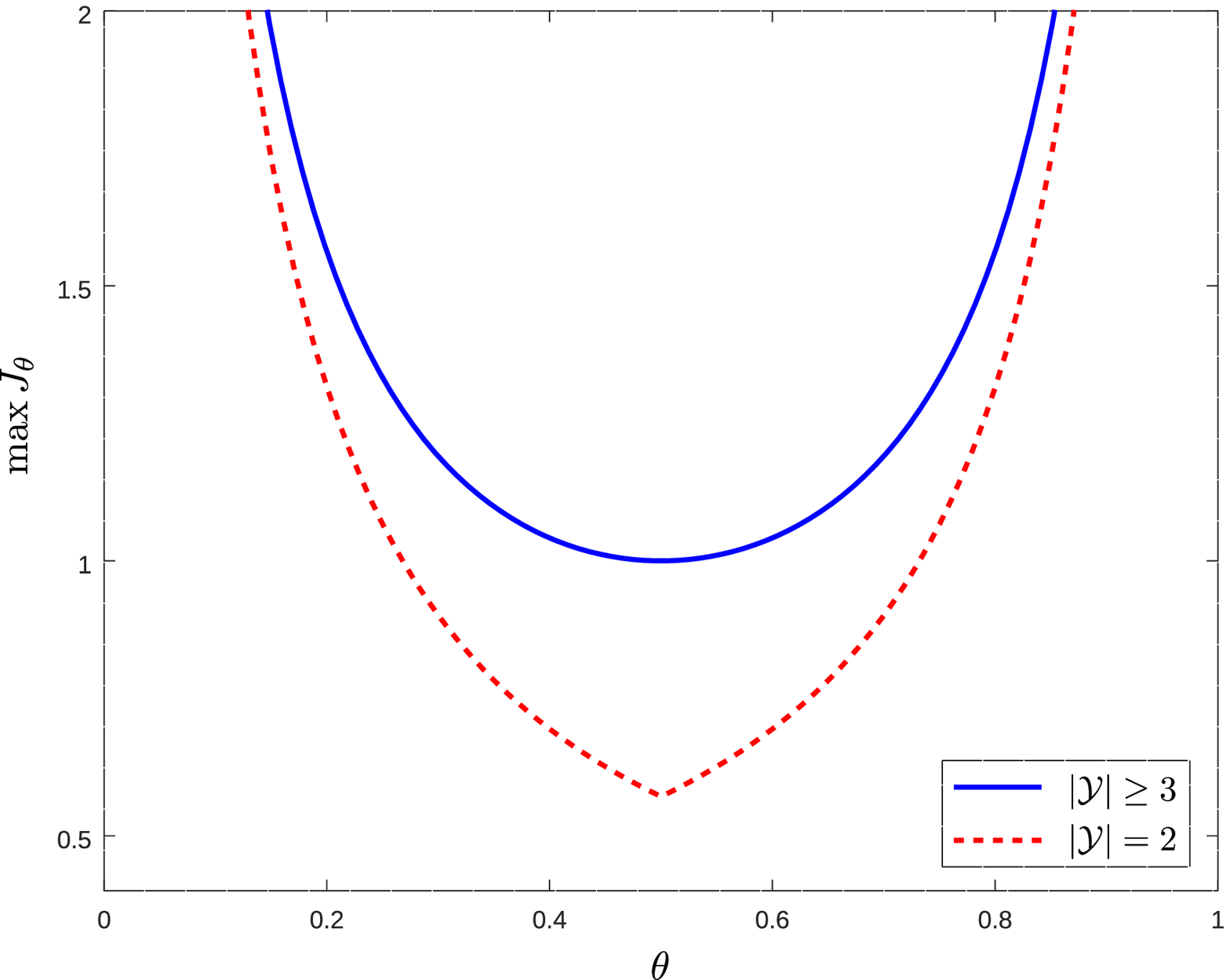}
	\vskip2ex
	\includegraphics[scale=0.5]{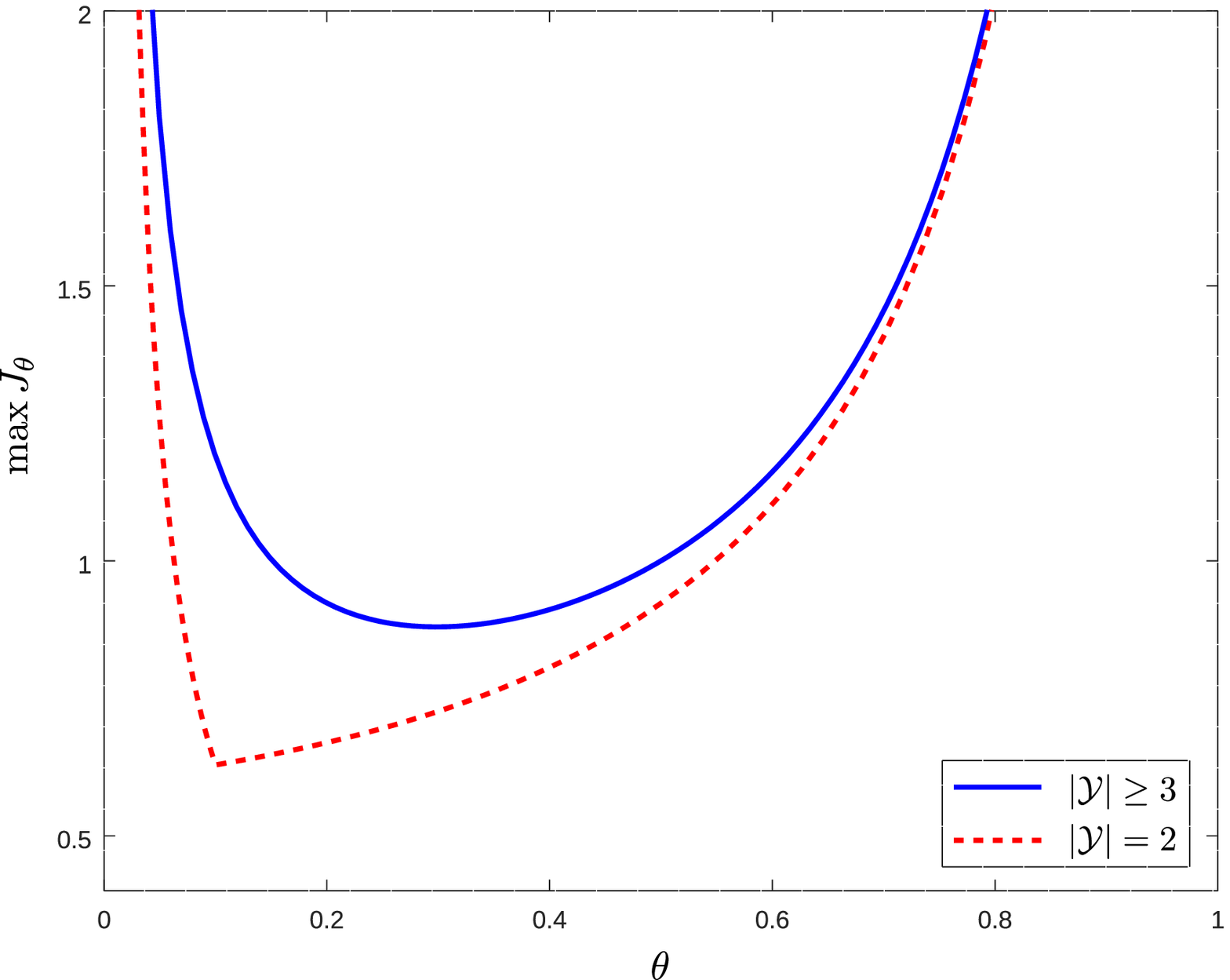}
	\caption{The blue solid curves and the red broken curves illustrate the maximum values in Theorem~\ref{c-var} 
	when $|\mathcal{Y}|\ge3$ and $|\mathcal{Y}|=2$, respectively. 
	The upper and lower figures illustrate the cases $(\delta,w)=(1/4,1/2)$ and $(\delta,w)=(1/4,2/5)$, respectively.}
	\label{F2}
\end{figure}

Pairs of two distributions in the cases (iii) contain all pairs in the case $|\mathcal{Y}|=2$. 
Hence the maximum Fisher information can be achieved 
at least when the number of elements of the probability space $\mathcal{Y}$ is three.
In this case, the optimal choice of two distributions $p_0$ and $p_1$ does not depend on $\theta$.
However, it depends on $\delta$ and $w$. 
Fig.~\ref{F2} illustrates the relation between the parameter $\theta$ and the maximum value given in Theorem~\ref{c-var} 
when the threshold $\delta$ and the weight $w$ take the specific values. 
Note that in this figure, one can readily identify the difference between 
the case (iii) $|\mathcal{Y}|\ge3$ and the case $|\mathcal{Y}|=2$ which is the combination of (i) and (ii).

\section{Optimization with general sublinear function}\label{S3}
To prove Theorem~\ref{c-var} when $|\mathcal{Y}|=3$, we maximize a more general objective function, 
which is the sum of values of a \textit{sublinear function}. 
Although this objective function was optimized in \cite{KOV16} under the $(\epsilon,0)$-differential privacy constraint, 
we optimize it under the $l^1$-norm constraint. 
First, we define sublinear functions as follows.

\begin{definition}\label{dfn:subl}
	A function $\psi: [0,\infty)^2\to\mathbb{R}$ is sublinear 
	if the following conditions hold: 
	\begin{itemize}
		\item
		$\psi(\alpha x,\alpha y) = \alpha\psi(x,y)$ for all $\alpha>0$ and $x,y\ge0$; 
		\item
		$\psi(x_1+x_2,y_1+y_2) \le \psi(x_1,y_1) + \psi(x_2,y_2)$ for all $x_1,x_2,y_1,y_2\ge0$.
	\end{itemize}
\end{definition}

Any sublinear function is convex. 
Thus the sum of values of a sublinear function is also convex. 
Conversely, if a convex function $\psi$ satisfies the first condition in Definition~\ref{dfn:subl}, 
then $\psi$ is sublinear. This fact follows from Definition~\ref{dfn:subl} immediately.

Let $\psi$ be a sublinear function and define the function $\Psi$ as 
\begin{equation}
	\Psi(p_0,p_1) := \sum_{y\in\mathcal{Y}} \psi(p_0(y),p_1(y)) \label{LPR}
\end{equation}
for any two distributions $p_0$ and $p_1$. 
Then we can maximize $\Psi$ under the $l^1$-norm constraint as follows.

\begin{theorem}\label{thm:subl}
	When the conditions $|\mathcal{Y}|\ge3$, $w\in[a,1-a]$, and $\delta\in(0,1)$ hold, 
	the maximization problem 
	\begin{equation}
		\max_{\|(1-w)p_0 - w p_1\|_1\le\delta} \Psi(p_0,p_1) \label{eq:subl1}
	\end{equation}
	achieves the maximum value at the ordered pair $(\bar{p}_0,\bar{p}_1)$ of the two distributions 
	\[
	\bar{p}_0 = \Bigl[ \frac{a}{1-w}, 1-\frac{a}{1-w}, 0 \Bigr],\quad
	\bar{p}_1 = \Bigl[ \frac{a}{w}, 0, 1-\frac{a}{w} \Bigr].
	\]
\end{theorem}
\begin{proof}
	We regard the probability space $\mathcal{Y}$ as the set $\{ 1,\ldots,N \}$. 
	To show this theorem, we remark a few facts. 
	First, the domain of the maximization problem \eqref{eq:subl1} is compact and convex. 
	Second, the objective function is convex. 
	Thus the maximum value is achieved at an extreme point of the domain. 
	We assume that $(p_0,p_1)$ is an extreme point of the domain in the following steps.
	\par\setcounter{count}{0}
	\noindent\textbf{Step~\num.}
	We show that there exists an element $y\in\mathcal{Y}$ satisfying $p_0(y)p_1(y) > 0$ by contradiction. 
	Suppose that any element $y\in\mathcal{Y}$ satisfies $p_0(y)p_1(y) = 0$. 
	This assumption implies $\delta = \|(1-w)p_0 - w p_1\|_1 = (1-w) + w = 1$, 
	which contradicts the assumption $0<\delta<1$. 
	Therefore, an element $y\in\mathcal{Y}$ satisfies $p_0(y)p_1(y) > 0$. 
	Without loss of generality, we may assume that the above element $y$ is $1$, i.e., $p_0(1)p_1(1) > 0$ by changing elements if necessarily.
	\par
	\noindent\textbf{Step~\num.}
	We prove that any element $y\not=1$ satisfies $p_0(y)p_1(y) = 0$ by contradiction. 
	Suppose $p_0(2)p_1(2) > 0$ by changing elements if necessarily. 
	Taking a sufficiently small positive number $\epsilon$, 
	we define the distributions $q_k$ and $q'_k$ for $k=0,1$ as 
	\begin{gather*}
		q_0(1) = p_0(1) - w\epsilon,\quad q'_0(1) = p_0(1) + w\epsilon,\\
		q_0(2) = p_0(2) + w\epsilon,\quad q'_0(2) = p_0(2) - w\epsilon,\\
		q_0(y) = q'_0(y) = p_0(y) \quad (3\le y\le N),\\
		q_1(1) = p_1(1) - (1-w)\epsilon,\quad q'_1(1) = p_1(1) + (1-w)\epsilon,\\
		q_1(2) = p_1(2) + (1-w)\epsilon,\quad q'_1(2) = p_1(2) - (1-w)\epsilon,\\
		q_1(y) = q'_1(y) = p_1(y) \quad (3\le y\le N).
	\end{gather*}
	Then, for any $k=0,1$, the relations 
	$\|(1-w)q_0 - w q_1\|_1 = \|(1-w)q'_0 - w q'_1\|_1 \le \delta$, 
	$p_k = (q_k + q'_k)/2$, and $q_k \not= q'_k$ hold, 
	which contradicts that the point $(p_0,p_1)$ is an extreme point of the domain. 
	Thus any element $y\not=1$ satisfies $p_0(y)p_1(y) = 0$.
	\par
	\noindent\textbf{Step~\num.}
	We show the relations 
	\begin{gather}
		\min\{ (1-w)p_0(1), w p_1(1) \} \ge a,\label{eq:min}\\
		\begin{split}
			\Psi(p_0,p_1) &= \psi(p_0(1),p_1(1))\\
			&\quad+ (1-p_0(1))\psi(1,0) + (1 - p_1(1))\psi(0,1).
		\end{split}\label{eq:subl2}
	\end{gather}
	Using the result in Step~2, for each element $y\not=1$, 
	we can take an element $k(y)\in\mathbb{Z}/2\mathbb{Z}$ satisfying $p_{k(y)+1}(y)=0$. 
	Then the inequality 
	\begin{align*}
		\delta &\ge \|(1-w)p_0 - w p_1\|_1\\
		&= |(1-w)p_0(1) - w p_1(1)|\\
		&\quad+ (1-w)\sum_{\scriptsize
		\begin{array}{c}
			2\le y\le N\\
			k(y)=0
		\end{array}
		} p_0(y) + w\sum_{\scriptsize
		\begin{array}{c}
			2\le y\le N\\
			k(y)=1
		\end{array}
		} p_1(y)\\
		&= |(1-w)p_0(1) - w p_1(1)|\\
		&\quad+ (1-w)(1 - p_0(1)) + w(1 - p_1(1))\\
		&= |(1-w)p_0(1) - w p_1(1)| + 1 - (1-w)p_0(1) - w p_1(1)
	\end{align*}
	holds, which means \eqref{eq:min}. Further, \eqref{eq:subl2} is verified as follows: 
	\begin{align*}
		\Psi(p_0,p_1) &= \psi(p_0(1),p_1(1)) + \sum_{\scriptsize
		\begin{array}{c}
			2\le y\le N\\
			k(y)=0
		\end{array}
		} p_0(y)\psi(1,0)\\
		&\quad+ \sum_{\scriptsize
		\begin{array}{c}
			2\le y\le N\\
			k(y)=1
		\end{array}
		} p_1(y)\psi(0,1)\\
		&= \psi(p_0(1),p_1(1))\\
		&\quad+ (1-p_0(1))\psi(1,0) + (1 - p_1(1))\psi(0,1).
	\end{align*}
	\par
	\noindent\textbf{Step~\num.}
	In this step, assuming $N\ge3$, we show our assertion. 
	The sublinearity of $\psi$ implies 
	\begin{align}
		&\Psi(\bar{p}_0,\bar{p}_1) - \Psi(p_0,p_1)\nonumber\\
		\overset{(a)}{=}& \psi\Bigl( \frac{a}{1-w}, \frac{a}{w} \Bigr) + \Bigl( 1-\frac{a}{1-w} \Bigr)\psi(1,0) + \Bigl( 1-\frac{a}{w} \Bigr)\psi(0,1)\nonumber\\
		&- \psi(p_0(1),p_1(1))\nonumber\\
		&- (1-p_0(1))\psi(1,0) - (1-p_1(1))\psi(0,1)\nonumber\\
		=& \psi\Bigl( \frac{a}{1-w}, \frac{a}{w} \Bigr) - \psi(p_0(1),p_1(1))\nonumber\\
		&+ \Bigl( p_0(1)-\frac{a}{1-w} \Bigr)\psi(1,0) + \Bigl( p_1(1)-\frac{a}{w} \Bigr)\psi(0,1)\nonumber\\
		\ge& \psi\Bigl( \frac{a}{1-w}, \frac{a}{w} \Bigr) - \psi\Bigl( \frac{a}{1-w}, \frac{a}{w} \Bigr)\nonumber\\
		&- \psi\Bigl( p_0(1)-\frac{a}{1-w}, 0 \Bigr) - \psi\Bigl( 0, p_1(1)-\frac{a}{w} \Bigr)\nonumber\\
		&+ \Bigl( p_0(1)-\frac{a}{1-w} \Bigr)\psi(1,0) + \Bigl( p_1(1)-\frac{a}{w} \Bigr)\psi(0,1)\label{eq:fisher1}\\
		\overset{(b)}{=}& 0\nonumber,
	\end{align}
	where \eqref{eq:subl2} and \eqref{eq:min} have been used to obtain $(a)$ and $(b)$, respectively. 
	Therefore, $\Psi(\bar{p}_0,\bar{p}_1)$ is the maximum value of the maximization problem \eqref{eq:subl1}.
\end{proof}

\begin{proof}[Proof of Theorem~\ref{c-var}]
	To use Theorem~\ref{thm:subl}, we check a few facts. 
	First, the Fisher information $J_\theta(p_0,p_1)$ can be written as 
	\begin{equation}
		J_\theta(p_0,p_1) = \sum_{y\in\mathcal{Y}} \frac{(p_1(y) - p_0(y))^2}{p_\theta(y)}. \label{eq:fisher2}
	\end{equation}
	Second, the function $(x,y)\mapsto x^2/y$ is convex \cite[Section~3.2.6]{Boyd} and 
	thus the function $\psi: [0,\infty)^2\to\mathbb{R}$ defined as 
	\[
	\psi(x,y) = 
	\begin{dcases}
		\frac{(y-x)^2}{(1-\theta)x+\theta y} & x+y>0,\\
		0 & x=y=0,
	\end{dcases}
	\]
	is sublinear. 
	Hence the objective function can be written as $J_\theta(p_0,p_1) = \sum_{y\in\mathcal{Y}} \psi(p_0(y),p_1(y))$. 
	Therefore, Theorem~\ref{thm:subl} implies Theorem~\ref{c-var} when $|\mathcal{Y}|\ge3$.
	\par
	We show Theorem~\ref{c-var} when $|\mathcal{Y}|=2$. 
	In this case, the maximum value is also achieved at an extreme point of the domain; 
	Steps~1--3 in the proof of Theorem~\ref{thm:subl} also hold. 
	Let $(p_0,p_1)$ be an extreme point of the domain. 
	Then one of the following two cases occurs: 
	\begin{gather}
		a/w\le\exists x\le1,\quad p_0 = [1,0],\quad p_1 = [x,1-x],\nonumber\\
		J_\theta(p_0,p_1) = \frac{(1-x)^2}{1-\theta+\theta x} + \frac{1-x}{\theta}
		= \frac{1-x}{\theta(1-\theta+\theta x)} \label{eq:fisher3}
	\end{gather}
	and 
	\begin{gather}
		a/(1-w)\le\exists x\le1,\quad p_0 = [x,1-x],\quad p_1 = [1,0],\nonumber\\
		\begin{split}
			J_\theta(p_0,p_1) &= \frac{(1-x)^2}{(1-\theta)x + \theta} + \frac{1-x}{1-\theta}\\
			&= \frac{1-x}{(1-\theta)\{ (1-\theta)x + \theta \}},
		\end{split}\label{eq:fisher4}
	\end{gather}
	where note \eqref{eq:min}. 
	The values \eqref{eq:fisher3} and \eqref{eq:fisher4} strictly decrease as $x$ increases. 
	Thus we obtain $x=a/w$ and $x=a/(1-w)$ in the first and second cases, respectively. 
	Then the difference between the right-hand sides of \eqref{eq:fisher3} and \eqref{eq:fisher4} equals 
	\begin{align*}
		&\frac{1-a/w}{\theta(1-\theta+\theta a/w)} - \frac{1-a/(1-w)}{(1-\theta)\{ (1-\theta)a/(1-w) + \theta \}}\\
		=& \frac{w-a}{\theta\{ (1-\theta)w+\theta a \}} - \frac{1-w-a}{(1-\theta)\{ (1-\theta)a + \theta(1-w) \}}\\
		=& \bigl[ (w-a)(1-\theta)\{ a(1-\theta) + (1-w)\theta \}\\
		&- (1-w-a)\theta\{ w(1-\theta) + a\theta \} \bigr]\\
		&\quad\Big/ \theta(1-\theta)\{ (1-\theta)w+\theta a \}\{ (1-\theta)a + \theta(1-w) \}.
	\end{align*}
	This numerator is calculated as follows: 
	\begin{align*}
		&(w-a)(1-\theta)\{ a(1-\theta) + (1-w)\theta \}\\
		&- (1-w-a)\theta\{ w(1-\theta) + a\theta \}\\
		=& \theta(1-\theta)\{(w-a)(1-w) - (1-w-a)w\}\\
		&+ (w-a)a(1-\theta)^2 - (1-w-a)a\theta^2\\
		=& \theta(1-\theta)(2w-1)a + \{ w(1-\theta)^2 - (1-w)\theta^2 \}a\\
		&+ \{ \theta^2 - (1-\theta)^2 \}a^2\\
		=& \theta(1-\theta)(2w-1)a + \{ w(2\theta^2 - 2\theta + 1) - \theta^2 \}a\\
		&+ (2\theta-1)a^2\\
		=& (w-\theta)a + (2\theta-1)a^2 = a\{ (2a-1)\theta+w-a \}\\
		=& a(-\delta\theta+w-a) = a\delta(\theta_0-\theta).
	\end{align*}
	Therefore, Theorem~\ref{c-var} has been proved.
\end{proof}

\section{Dependence on weight}\label{S4}
When $|\mathcal{Y}|\ge3$, the optimal pair in Theorem~\ref{c-var} depends on the weight $w$. 
We are interested in whether we can take an optimal pair that is independent of the weight $w$. 
Strictly speaking, we are interested in whether there exists an ordered pair $(p_0,p_1)$ such that 
for any weight $w\in[a,1-a]$ the value $J_\theta(p_0,p_1)$ is the maximum value. 
The following theorem answers this question.

\begin{theorem}\label{opt-s}
	Assume $|\mathcal{Y}|\ge3$. 
	If an ordered pair $(p_0,p_1)$ of two distributions achieves the maximum value in Theorem~\ref{c-var}, 
	then any element $y\in\mathcal{Y}$ with $p_0(y)p_1(y)>0$ satisfies 
	\begin{equation}
		(1-w)p_0(y)=w p_1(y). \label{eq:opt-s1}
	\end{equation}
	Moreover, the above ordered pair $(p_0,p_1)$ satisfies 
	\begin{equation}
		(1-w)\sum_{p_1(y)>0} p_0(y) = w\sum_{p_0(y)>0} p_1(y) = a. \label{eq:opt-s2}
	\end{equation}
	Conversely, if an ordered pair $(p_0,p_1)$ satisfies the above two conditions, 
	then it achieves the maximum value in Theorem~\ref{c-var}.
\end{theorem}

Indeed, the pair $(\bar{p}_0,\bar{p}_1)$ given in Theorem~\ref{thm:subl} satisfies the two conditions in Theorem~\ref{opt-s}. 
Actually, when $|\mathcal{Y}|=3$, the pair $(\bar{p}_0,\bar{p}_1)$ is a unique optimal pair up to rearrangement of ($p_0$ and $p_1$'s) components. 
The case $|\mathcal{Y}|>3$ yields some degrees of freedom. 
In fact, all optimal pairs $(p_0,p_1)$ can be written as 
\begin{align*}
	p_0 &= \Bigl[ \underbrace{\frac{a}{1-w}b_1,\ldots,\frac{a}{1-w}b_{r_1}}_{r_1},\\
	&\qquad\Bigl( \underbrace{1-\frac{a}{1-w} \Bigr)b_{r_1+1},\ldots,\Bigl( 1-\frac{a}{1-w} \Bigr)b_{r_2}}_{r_2-r_1},\\
	&\qquad\qquad\underbrace{0,\ldots,0}_{r_3-r_2}, \underbrace{0,\ldots,0}_{|\mathcal{Y}|-r_3} \Bigr],\\
	p_1 &= \Bigl[ \underbrace{\frac{a}{w}b_1,\ldots,\frac{a}{w}b_{r_1}}_{r_1}, 
	\underbrace{0,\ldots,0}_{r_2-r_1},\\
	&\qquad\underbrace{\Bigl( 1-\frac{a}{w} \Bigr)b_{r_2+1},\ldots,\Bigl( 1-\frac{a}{w} \Bigr)b_{r_3}}_{r_3-r_2}, 
	\underbrace{0,\ldots,0}_{|\mathcal{Y}|-r_3} \Bigr]
\end{align*}
up to rearrangement of ($p_0$ and $p_1$'s) components, 
where $b_1,\ldots,b_{|\mathcal{Y}|}$ are arbitrary non-negative numbers satisfying 
$\sum_{y=1}^{r_1} b_y = \sum_{y=r_1+1}^{r_2} b_y = \sum_{y=r_2+1}^{r_3} b_y = \sum_{y=r_3+1}^{|\mathcal{Y}|} b_y = 1$, and 
$1\le r_1<r_2<r_3\le|\mathcal{Y}|$ are arbitrary integers. 
Hence the parameters $b_1,\ldots,b_{|\mathcal{Y}|}$, $r_1$, $r_2$, and $r_3$ describe degrees of freedom in the case $|\mathcal{Y}|>3$.

Further, Theorem~\ref{opt-s} yields the following corollary immediately.

\begin{corollary}\label{coro1}
	Assume $|\mathcal{Y}|\ge3$. 
	Then, for any two distinct weights. no ordered pair of two distributions achieves the maximum value in Theorem~\ref{c-var}.
\end{corollary}
\begin{proof}[Proof of Corollary~\ref{coro1}]
	Fixing a threshold $\delta$, 
	we assume that an ordered pair $(p_0,p_1)$ of two distributions achieves the maximum value in Theorem~\ref{c-var} for some weight $w$. 
	Then Theorem~\ref{opt-s} yields $\sum_{p_0(y)>0} p_1(y)=a/w$. 
	Since any weight $w'\not=w$ satisfies $\sum_{p_0(y)>0} p_1(y)=a/w\not=a/w'$, 
	Theorem~\ref{opt-s} implies that the ordered pair $(p_0,p_1)$ does not achieve the maximum value in Theorem~\ref{c-var} for the weight $w'$. 
	Therefore, for any two distinct weights, no ordered pair of two distributions achieves the maximum value in Theorem~\ref{c-var}.
\end{proof}

Now, in order to prove Theorem~\ref{opt-s}, 
we need the equality condition of the convexity inequality of $J_\theta$. 
That is, we need the following lemma.
 
\begin{lemma}\label{conv-eq}
	Let $q_k$ and $q'_k$ with $k=0,1$ be distributions on $\mathcal{Y}$, and 
	$\theta$ and $t$ be real numbers with $0<\theta,t<1$. 
	Then the equation 
	\begin{equation*}
		J_\theta((1-t)q_0 + t q'_0, (1-t)q_1 + t q'_1)
		= (1-t)J_\theta(q_0,q_1) + t J_\theta(q'_0,q'_1)
	\end{equation*}
	holds if and only if any element $y\in\mathcal{Y}$ satisfies 
	\begin{equation}
		\det
		\begin{bmatrix}
			q_1(y) & q_0(y)\\
			q'_1(y) & q'_0(y)
		\end{bmatrix}
		= 0. \label{eq:fisher5}
	\end{equation}
\end{lemma}
\begin{proof}
	First, we see the equality condition of the convexity inequality of the convex function $g(x,y) := x^2/y$ 
	defined for any real number $x$ and any positive number $y$. 
	Let $0<t<1$, $x_i\in\mathbb{R}$, and $y_i>0$ for $i=0,1$. 
	Putting $x_t := (1-t)x_0 + t x_1$ and $y_t := (1-t)y_0 + t y_1$, we have 
	\begin{align*}
		&(1-t)g(x_0,y_0) + tg(x_1,y_1) - g(x_t,y_t)\\
		=& \{ (1-t)x_0^2 y_1 y_t + tx_1^2 y_0 y_t - x_t^2 y_0 y_1 \} / y_0 y_1 y_t\\
		=& \{ (1-t)^2 x_0^2 y_0 y_1 + (1-t)t x_0^2 y_1^2 + t(1-t) x_1^2 y_0^2 + t^2 x_1^2y_0 y_1\\
		&- (1-t)^2 x_0^2 y_0 y_1 - t^2 x_1^2 y_0 y_1 - 2t(1-t) x_0 x_1 y_0 y_1 \} / y_0 y_1 y_t\\
		=& t(1-t)(x_0^2 y_1^2 + x_1^2 y_0^2 - 2x_0 x_1 y_0 y_1) / y_0 y_1 y_t\\
		=& t(1-t)(x_0 y_1 - x_1 y_0)^2 / y_0 y_1 y_t \ge 0.
	\end{align*}
	Therefore, the equation $g(x_t,y_t) = (1-t)g(x_0,y_0) + tg(x_1,y_1)$ holds if and only if 
	\[
	\det
	\begin{bmatrix}
		x_0 & y_0\\
		x_1 & y_1
	\end{bmatrix}
	= 0.
	\]
	\par
	Next, we see the equality condition of the convexity inequality of the convex function $f$ defined as 
	\begin{equation*}
		f(x,y) := 
		\begin{dcases}
			\frac{(y-x)^2}{(1-\theta)x + \theta y} & x+y>0,\\
			0 & x=y=0.
		\end{dcases}
	\end{equation*}
	Let $0<t<1$ and $x_i,y_i\ge0$ for $i=0,1$, and 
	let $x_t := (1-t)x_0 + tx_1$ and $y_t := (1-t)y_0 + ty_1$. 
	First, assume $x_i+y_i>0$ for $i=0,1$. 
	Noting the equality condition of the convexity inequality of $g$, we find that 
	the equation $f(x_t,y_t) = (1-t)f(x_0,y_0) + t f(x_1,y_1)$ holds if and only if 
	\[
	\det
	\begin{bmatrix}
		y_0-x_0 & (1-\theta)x_0 + \theta y_0\\
		y_1-x_1 & (1-\theta)x_1 + \theta y_1
	\end{bmatrix}
	= 0.
	\]
	The left-hand side of this equation equals 
	\[
	\det
	\begin{bmatrix}
		y_0-x_0 & x_0\\
		y_1-x_1 & x_1
	\end{bmatrix}
	= 
	\det
	\begin{bmatrix}
		y_0 & x_0\\
		y_1 & x_1
	\end{bmatrix}.
	\]
	Thus the equation $f(x_t,y_t) = (1-t)f(x_0,y_0) + t f(x_1,y_1)$ holds if and only if 
	\begin{equation}
		\det
		\begin{bmatrix}
			y_0 & x_0\\
			y_1 & x_1
		\end{bmatrix}
		= 0. \label{eq:f}
	\end{equation}
	Second, when $x_0=y_0=0$ or $x_1=y_1=0$, 
	it can be easily checked that the equations $f(x_t,y_t) = (1-t)f(x_0,y_0) + t f(x_1,y_1)$ and \eqref{eq:f} hold. 
	Therefore, summarizing the above two cases, we find that 
	the equation $f(x_t,y_t) = (1-t)f(x_0,y_0) + t f(x_1,y_1)$ holds if and only if 
	the equation \eqref{eq:f} holds.
	\par
	Finally, we show this lemma. 
	Since the Fisher information $J_\theta(p_0,p_1)$ can be written as \eqref{eq:fisher2}, 
	the equality condition \eqref{eq:fisher5} of the convexity inequality of $J_\theta$ follows from that of $f$.
\end{proof}

\begin{proof}[Proof of Theorem~\ref{opt-s}]
	Assume that an ordered pair $(p_0,p_1)$ of two distributions achieves the maximum value in Theorem~\ref{c-var} when $|\mathcal{Y}|\ge3$.
	\par\setcounter{count}{0}
	\noindent\textbf{Step~\num.}
	In the same way as Step~1 in the proof of Theorem~\ref{thm:subl}, 
	it is shown that there exists an element $y\in\mathcal{Y}$ such that $p_0(y)p_1(y)>0$. 
	By changing elements if necessarily, we may assume that 
	there exist three elements $r_1$, $r_2$, and $r_3$ with $1 \le r_1 \le r_2 \le r_3 \le N$ such that 
	\begin{gather*}
		p_0(y)p_1(y) > 0\quad(1 \le y \le r_1),\\
		p_0(y) > 0,\quad p_1(y) = 0\quad(r_1 < y \le r_2),\\
		p_0(y) = 0,\quad p_1(y) > 0\quad(r_2 < y \le r_3),\\
		p_0(y) = 0,\quad p_1(y) = 0\quad(r_3 < y \le N).
	\end{gather*}
	Then \eqref{eq:opt-s1} and \eqref{eq:opt-s2} turn to 
	\begin{gather}
		(1-w)p_0(y)=w p_1(y)\quad(1\le y\le r_1),\label{eq:opt-s1'}\\
		(1-w)\sum_{y=1}^{r_1} p_0(y) = w\sum_{y=1}^{r_1} p_1(y) = a,\label{eq:opt-s2'}
	\end{gather}
	respectively.
	\par
	\noindent\textbf{Step~\num.}
	In this step, assuming $r_1\ge2$, we show \eqref{eq:opt-s1'}. 
	Define the distributions $q_k$ and $q'_k$ for $k=0,1$ in the same way as Step~2 in the proof of Theorem~\ref{thm:subl}. 
	Then, for any $k=0,1$, the relations 
	\begin{gather*}
		\|(1-w)q_0 - w q_1\|_1 = \|(1-w)q'_0 - w q'_1\|_1 \le \delta,\\
		p_k = (q_k + q'_k)/2,\quad q_k \not= q'_k
	\end{gather*}
	hold. 
	Since the function $J_\theta$ is convex and the value $J_\theta(p_0,p_1)$ is the maximum value, we have 
	\begin{equation*}
		J_\theta(p_0,p_1) \le \frac{1}{2}J_\theta(q_0,q_1) + \frac{1}{2}J_\theta(q'_0,q'_1) \le J_\theta(p_0,p_1).
	\end{equation*}
	Then Lemma~\ref{conv-eq} implies 
	\begin{align*}
		0 &= \det
		\begin{bmatrix}
			q_1(1) & q_0(1)\\
			q'_1(1) & q'_0(1)
		\end{bmatrix}
		= \det
		\begin{bmatrix}
			q_1(1) & q_0(1)\\
			q'_1(1) - q_1(1) & q'_0(1) - q_0(1)
		\end{bmatrix}
		\\
		&= \det
		\begin{bmatrix}
			q_1(1) & q_0(1)\\
			2(1-w)\epsilon & 2w\epsilon
		\end{bmatrix}
		= 2\epsilon\det
		\begin{bmatrix}
			p_1(1) & p_0(1)\\
			1-w & w
		\end{bmatrix}
		\\
		&= -2\epsilon\{ (1-w)p_0(1) - w p_1(1) \},
	\end{align*}
	whence $(1-w)p_0(1) - w p_1(1) = 0$. 
	Lemma~\ref{conv-eq} also implies 
	\begin{align*}
		0 &= \det
		\begin{bmatrix}
			q_1(2) & q_0(2)\\
			q'_1(2) & q'_0(2)
		\end{bmatrix}
		= \det
		\begin{bmatrix}
			q_1(2) & q_0(2)\\
			q'_1(2) - q_1(2) & q'_0(2) - q_0(2)
		\end{bmatrix}
		\\
		&= \det
		\begin{bmatrix}
			q_1(2) & q_0(2)\\
			-2(1-w)\epsilon & -2w\epsilon
		\end{bmatrix}
		= -2\epsilon\det
		\begin{bmatrix}
			p_1(2) & p_0(2)\\
			1-w & w
		\end{bmatrix}
		\\
		&= 2\epsilon\{ (1-w)p_0(2) - w p_1(2) \},
	\end{align*}
	whence $(1-w)p_0(2) - w p_1(2) = 0$. 
	Replacing the above element $2$ with $y=3,\ldots,r_1$, 
	we can show the equation $(1-w)p_0(y) - w p_1(y) = 0$ in the same way.
	\par
	\noindent\textbf{Step~\num.}
	We show \eqref{eq:opt-s2'}. 
	Assume $r_1=1$. In this case, note that \eqref{eq:opt-s2'} contains \eqref{eq:opt-s1'}. 
	In the same way as Step~3 in the proof of Theorem~\ref{thm:subl}, 
	the relations \eqref{eq:min} and \eqref{eq:subl2} hold. 
	Then \eqref{eq:fisher1} must be the equality 
	because the value $J_\theta(p_0,p_1)=\Psi(p_0,p_1)$ equals the maximum value $J_\theta(\bar{p}_0,\bar{p}_1)=\Psi(\bar{p}_0,\bar{p}_1)$. 
	Hence the equations 
	\begin{gather*}
		\psi(p_0(1),p_1(1))
		= \psi\Bigl(  p_0(1), \frac{a}{w} \Bigr) + \psi\Bigl( 0, p_1(1)-\frac{a}{w} \Bigr),\\
		\psi\Bigl(  p_0(1), \frac{a}{w} \Bigr)
		= \psi\Bigl( \frac{a}{1-w}, \frac{a}{w} \Bigr) + \psi\Bigl( p_0(1)-\frac{a}{1-w}, 0 \Bigr)
	\end{gather*}
	follow, where $\psi$ equals $f$ in the proof of Lemma~\ref{conv-eq}. 
	Using the equality condition of the convexity inequality of $\psi$ and noting that $\psi$ is sublinear, we obtain 
	\begin{gather*}
		0 = \det
		\begin{bmatrix}
			a/w&p_0(1)\\
			p_1(1)-a/w&0
		\end{bmatrix}
		= -p_0(1)(p_1(1)-a/w),\\
		\begin{split}
			0 &= \det
			\begin{bmatrix}
				a/w&a/(1-w)\\
				0&p_0(1)-a/(1-w)
			\end{bmatrix}
			\\
			&= (a/w)(p_0(1)-a/(1-w)),
		\end{split}
	\end{gather*}
	whence $(1-w)p_0(1)=w p_1(1)=a$.
	\par
	Assume $r_1\ge2$. The result in Step~2 and the sublinearity of $\psi$ imply 
	\begin{align*}
		&\Psi(p_0,p_1)\\
		=& \psi\Bigl( \sum_{y=1}^{r_1} p_0(y), \sum_{y=1}^{r_1} p_1(y) \Bigr)\\
		&+ \Bigl( 1-\sum_{y=1}^{r_1} p_0(y) \Bigr)\psi(1,0)
		+ \Bigl( 1-\sum_{y=1}^{r_1} p_1(y) \Bigr)\psi(0,1).
	\end{align*}
	Hence the equation $(1-w)\sum_{y=1}^{r_1} p_0(y) = w\sum_{y=1}^{r_1} p_1(y) = a$ 
	follows from the same argument in the previous paragraph. (Replace $p_k(1)$ with $\sum_{y=1}^{r_1} p_k(y)$.) 
	Summarizing the two cases $r_1=1$ and $r_1\ge2$, we obtain \eqref{eq:opt-s2'}.
	\par
	\noindent\textbf{Step~\num.}
	Finally, we show the converse part. 
	In this step, we use the same notations $\Psi$ and $\psi$ as Step~3. 
	Take an arbitrary ordered pair $(p_0,p_1)$ satisfying the two conditions in Theorem~\ref{opt-s}. 
	Then the converse part is verified as follows: 
	\begin{align*}
		&\Psi(p_0,p_1)\\
		=& \sum_{p_0(y)p_1(y)>0} \psi(p_0(y),p_1(y))\\
		&+ \sum_{p_0(y)p_1(y)=0} p_0(y)\psi(1,0) + \sum_{p_0(y)p_1(y)=0} p_1(y)\psi(0,1)\\
		\overset{(a)}{=}& \psi\Bigl( \sum_{p_0(y)p_1(y)>0} p_0(y), \sum_{p_0(y)p_1(y)>0} p_1(y) \Bigr)\\
		&+ \Bigl( 1-\sum_{p_0(y)p_1(y)>0} p_0(y) \Bigr)\psi(1,0)\\
		&+ \Bigl( 1-\sum_{p_0(y)p_1(y)>0} p_1(y) \Bigr)\psi(0,1)\\
		\overset{(b)}{=}& \psi\Bigl( \frac{a}{1-w}, \frac{a}{w} \Bigr)
		+ \Bigl( 1-\frac{a}{1-w} \Bigr)\psi(1,0) + \Bigl( 1-\frac{a}{w} \Bigr)\psi(0,1)\\
		=& \Psi(\bar{p}_0,\bar{p}_1),
	\end{align*}
	where \eqref{eq:opt-s1} and \eqref{eq:opt-s2} have been used to obtain $(a)$ and $(b)$, respectively.
\end{proof}

\section{Optimal discrimination}\label{S5}
Next, we assume that the true parameter is known to be either $\theta_1$ or $\theta_2$.
Under this assumption, we need to discriminate the two distributions
$p_{\theta_1}$ and $p_{\theta_2}$.
Then there are two kinds of error probabilities.
One is the error probability $P_{\theta_2\to \theta_1}$ that 
we incorrectly identity the parameter to be $\theta_1$ while the correct parameter is $\theta_2$.
The other is the error probability $P_{\theta_1\to \theta_2}$ that 
we incorrectly identity the parameter to be $\theta_2$ while the correct parameter is $\theta_1$.
When $n$ data $Y_1, \ldots, Y_n$ are observed,
we employ the likelihood ratio test: 
we support $p_{\theta_1}$  
when $\sum_{i=1}^n \ln(p_{\theta_1}(Y_i)/p_{\theta_2}(Y_i))$ is greater than a certain threshold; 
otherwise we support $p_{\theta_2}$.
The likelihood ratio test is known as the optimal method for this discrimination \cite[Chapter~2, Section~3]{Lehmann}. 
In the symmetric setting, we focus on the average of the above two error probabilities: 
$(1/2)P_{\theta_2\to \theta_1} + (1/2)P_{\theta_1\to \theta_2}$.
In this setting, when the threshold is chosen to be zero, 
the likelihood test realizes the minimum average error probability, 
which goes to zero exponentially as $n\to\infty$.
The exponentially decreasing rate is known as the Chernoff bound \cite{Chernoff,Cover} 
\begin{equation*}
	\sup_{-1<s<0} (-s)D_{1+s}(p_{\theta_1}\|p_{\theta_2}),
\end{equation*}
where the relative R\'{e}nyi entropy 
$D_{1+s}(p_{\theta_1}\|p_{\theta_2})$ 
is defined as
\begin{equation*}
	\begin{dcases}
		\frac{1}{s}\ln \sum_{y\in\mathcal{Y}} p_{\theta_1}(y)^{1+s} p_{\theta_2}(y)^{-s} & s\in(-1,0)\cup(0,\infty),\\
		\sum_{y\in\mathcal{Y}} p_{\theta_1}(y) \ln \frac{p_{\theta_1}(y)}{p_{\theta_2}(y)} & s=0.
	\end{dcases}
\end{equation*}
When $s=0$, it is simply called the relative entropy and denoted as $ D(p_{\theta_1}\|p_{\theta_2})$.
To guarantee the privacy for individual data, 
in the same way as Section~\ref{S2},
we impose a constant constraint on the $l^1$-norm
as $\| (1-w)p_0-w p_1 \|_1\le \delta$,
which is equivalent to the condition \eqref{KKO}.
Therefore, the value
\begin{align}
	&\max_{\|(1-w)p_0 - w p_1\|_1\le\delta} \sup_{-1<s<0} (-s)D_{1+s}(p_{\theta_1}\|p_{\theta_2})\nonumber\\
	=& \sup_{-1<s<0} (-s) \max_{\|(1-w)p_0 - w p_1\|_1\le\delta} D_{1+s}(p_{\theta_1}\|p_{\theta_2}) \label{eq:Chernoff}
\end{align}
expresses the maximum exponentially decreasing rate of the average of the two error probabilities 
under the above privacy condition.

We often focus on an asymmetric setting, in which, 
we impose the constant constraint on the error probability $P_{\theta_1\to \theta_2}$
and then minimize the other error probability $P_{\theta_2\to \theta_1}$.
In this case, the minimum value of $P_{\theta_2\to \theta_1}$ goes to zero exponentially.
As known as Stein's lemma \cite[Theorem~12.8.1]{Cover}, 
the maximum exponentially decreasing rate is known to be the relative entropy $D(p_{\theta_1}\|p_{\theta_2})$.
When we impose a constant constraint on the $l^1$-norm as $\| (1-w)p_0-w p_1 \|_1\le \delta$,
the maximum of the above exponent is 
$\max_{\|(1-w)p_0 - w p_1\|_1\le\delta} D(p_{\theta_1}\|p_{\theta_2})$.

As another setting, 
when we impose the exponentially decreasing rate of $P_{\theta_1\to \theta_2}$ to be greater than or equal to $r$,
the maximum exponentially decreasing rate of $P_{\theta_2\to \theta_1}$
is known \cite{Hoeffding} to be 
\begin{equation*}
	\sup_{-1<s<0} \frac{s}{1+s}\{ r-D_{1+s} (p_{\theta_2}\|p_{\theta_1}) \}.
\end{equation*}
This value is called the Hoeffding bound. 
Therefore, when we impose a constant constraint on the $l^1$-norm as $\| (1-w)p_0-w p_1 \|_1\le \delta$,
the maximum of the above exponent is 
\begin{align}
	&\max_{\|(1-w)p_0 - w p_1\|_1\le\delta} \sup_{-1<s<0} \frac{s}{1+s}\{ r-D_{1+s} (p_{\theta_2}\|p_{\theta_1}) \}\nonumber\\
	=&\sup_{-1<s<0} \frac{s}{1+s}\Bigl\{ r-\max_{\|(1-w)p_0 - w p_1\|_1\le\delta} D_{1+s} (p_{\theta_2}\|p_{\theta_1}) \Bigr\} \label{eq:Hoeffding}.
\end{align}

Further, if the exponentially decreasing rate of $P_{\theta_2\to \theta_1}$ 
is greater than the relative entropy $D(p_{\theta_1}\|p_{\theta_2})$,
the other error probability $P_{\theta_1\to \theta_2}$ goes to one \cite{Blahut}.
In this case, we often focus on the exponentially decreasing rate of 
$1-P_{\theta_2\to \theta_1}$, in which, a smaller exponent of this value is better.
When we impose the exponentially decreasing rate of $P_{\theta_1\to \theta_2}$ to be greater than or equal to $r$, 
the minimum exponentially decreasing rate of $1-P_{\theta_2\to \theta_1}$
is known \cite{Han} to be 
\begin{equation*}
	\sup_{s>0} \frac{s}{1+s}\{ r-D_{1+s} (p_{\theta_2}\|p_{\theta_1}) \}.
\end{equation*}
This value is called the Han-Kobayashi bound.
When similar to \eqref{eq:Hoeffding},
we impose a constant constraint on the $l^1$-norm as $\| (1-w)p_0-w p_1 \|_1\le \delta$,
the minimum of the above exponent is 
\begin{align}
 \min_{\|(1-w)p_0 - w p_1\|_1\le\delta} \sup_{s>0} \frac{s}{1+s}\{ r-D_{1+s} (p_{\theta_2}\|p_{\theta_1}) \}.
	\label{eq:Han-Kobayashi}
\end{align}
Since $\min_{\|(1-w)p_0 - w p_1\|_1\le\delta}$ and $\sup_{s>0}$
are optimizations with the opposite direction, further analysis is not so trivial.
	
First, we tackle the above maximization problems except for \eqref{eq:Han-Kobayashi}, which can be solved as a special case of Theorem~\ref{thm:subl}.
To address these problems, we introduce the $f$-divergence \cite{Csiszar} 
\[
D_f(p_{\theta_1}\|p_{\theta_2}) := 
\sum_{y\in\mathcal{Y}} p_{\theta_2}(y)f\Bigl( \frac{p_{\theta_1}(y)}{p_{\theta_2}(y)} \Bigr),
\]
where $f: (0,\infty)\to\mathbb{R}$ is a convex function.
When a convex function $f$ is chosen as 
$f(x) = -x^{1+s}$ for $s\in(-1,0)$ and $f(x) = x^{1+s}$ for $s\in(0,\infty)$,
the relative R\'{e}nyi entropy can be recovered: 
\begin{align}
	&D_{1+s}(p_{\theta_1}\|p_{\theta_2})\nonumber\\
	=&
	\begin{dcases}
		\frac{1}{s}\ln\{ -D_f(p_{\theta_1}\|p_{\theta_2}) \} & s\in(-1,0),\\
		\frac{1}{s}\ln D_f(p_{\theta_1}\|p_{\theta_2}) & s\in(0,\infty).
	\end{dcases}\label{eq:Renyi}
\end{align}
Moreover, when $f(x)=x\ln x$, 
the $f$-divergence $D_f(p_{\theta_1}\|p_{\theta_2})$ equals the relative entropy $D(p_{\theta_1}\|p_{\theta_2})$.

To apply Theorem~\ref{thm:subl}, we define the function $\psi: [0,\infty)^2\to\mathbb{R}$ as 
\begin{align}
	&\psi(x,y)\nonumber\\
	=&
	\begin{dcases}
		((1-\theta_2)x + \theta_2 y)f\Bigl( \frac{(1-\theta_1)x + \theta_1 y}{(1-\theta_2)x + \theta_2 y} \Bigr) & x+y>0,\\
		0 & x=y=0.
	\end{dcases}\label{LPF}
\end{align}
The convexity of $f$ implies that the function $(x,y)\mapsto y f(x/y)$ is also convex \cite[Section~3.2.6]{Boyd}. 
In addition, the function $\psi$ satisfies the first condition in Definition~\ref{dfn:subl}, $\psi$ is sublinear. 
Under the definition \eqref{LPF}, 
the function $\Psi$ defined in \eqref{LPR} equals the $f$-divergence $D_f(p_{\theta_1}\|p_{\theta_2})$. 
Hence Theorem~\ref{thm:subl} implies the following theorem.

\begin{theorem}\label{f-div}
	When $f$ is a convex function and the conditions $|\mathcal{Y}|\ge3$, $w\in[a,1-a]$, and $\delta,\theta_1,\theta_2\in(0,1)$ hold, 
	we have 
	\begin{align*}
		&\max_{\|(1-w)p_0 - w p_1\|_1\le\delta} D_f(p_{\theta_1}\|p_{\theta_2})\\
		=& a\Bigl( \frac{1-\theta_2}{1-w} + \frac{\theta_2}{w} \Bigr)
		f\Bigl( \frac{(1-\theta_1)w + \theta_1 (1-w)}{(1-\theta_2)w + \theta_2 (1-w)} \Bigr)\\
		&+ \Bigl( 1-\frac{a}{1-w} \Bigr)(1-\theta_2)f\Bigl( \frac{1-\theta_1}{1-\theta_2} \Bigr)\\
		&+ \Bigl( 1-\frac{a}{w} \Bigr)(1-\theta_2)f\Bigl( \frac{\theta_1}{\theta_2} \Bigr),
	\end{align*}
	which is achieved by the ordered pair $(\bar{p}_0,\bar{p}_1)$ given in Theorem~\ref{thm:subl}.
\end{theorem}

The relative R\'{e}nyi entropy $D_{1+s}(p_{\theta_1}\|p_{\theta_2})$ with $s\in(-1,0)\cup(0,\infty)$ 
is the composite function with a monotone function and the $f$-divergence (see \eqref{eq:Renyi}). 
Further, the relative R\'{e}nyi entropy $D_{1+s}(p_{\theta_1}\|p_{\theta_2})$ with $s=0$ is just the $f$-divergence (when $f(x)=x\ln x$). 
Thus we have the following corollary.

\begin{corollary}\label{Renyi}
	When the conditions $|\mathcal{Y}|\ge3$, $w\in[a,1-a]$, and $\delta,\theta_1,\theta_2\in(0,1)$ hold, 
	we have 
	\begin{align*}
		&\max_{\|(1-w)p_0 - w p_1\|_1\le\delta} D_{1+s}(p_{\theta_1}\|p_{\theta_2})\\
		=& 
		\begin{dcases}
			\frac{1}{s}\ln\Bigl[ 
			\frac{a}{w(1-w)} \frac{\{ (1-\theta_1)w + \theta_1 (1-w) \}^{1+s}}{\{ (1-\theta_2)w + \theta_2 (1-w) \}^{s}}\\
			\quad+ \Bigl( 1-\frac{a}{1-w} \Bigr)\frac{(1-\theta_1)^{1+s}}{(1-\theta_2)^{s}}
			+ \Bigl( 1-\frac{a}{w} \Bigr)\frac{\theta_1^{1+s}}{\theta_2^{s}}
			\Bigr]\\
			\hspace{5cm} s\in(-1,0)\cup(0,\infty),\\
			\frac{a}{w(1-w)} \{ (1-\theta_1)w + \theta_1 (1-w) \}\\
			\qquad\cdot\ln\frac{(1-\theta_1)w + \theta_1 (1-w)}{(1-\theta_2)w + \theta_2 (1-w)}\\
			\quad+ \Bigl( 1-\frac{a}{1-w} \Bigr)(1-\theta_1)\ln\frac{1-\theta_1}{1-\theta_2}
			+ \Bigl( 1-\frac{a}{w} \Bigr)\theta_1\ln\frac{\theta_1}{\theta_2}\\
			\hspace{5cm} s=0,
		\end{dcases}
	\end{align*}
	which is achieved by the ordered pair $(\bar{p}_0,\bar{p}_1)$ given in Theorem~\ref{thm:subl}.
\end{corollary}

Using the above corollary, we can optimize 
the Chernoff bound \eqref{eq:Chernoff} and the Hoeffding bound \eqref{eq:Hoeffding}
under the constraint $\|(1-w)p_0 - w p_1\|_1\le\delta$, by choosing the two distributions $\bar{p}_0$ and $\bar{p}_1$ given in Theorem~\ref{thm:subl}.

To understand Corollary~\ref{Renyi} visually, we have provided Fig.~\ref{F4}. 
Fig.~\ref{F4} illustrates the relation between the pair $(\theta_1,\theta_2)$ of the two parameters and the maximum value given in Corollary~\ref{Renyi} 
when the threshold $\delta$, the weight $w$, and the parameter $s$ take the specific values.

\begin{figure}[t]
	\centering
	\includegraphics[scale=0.48]{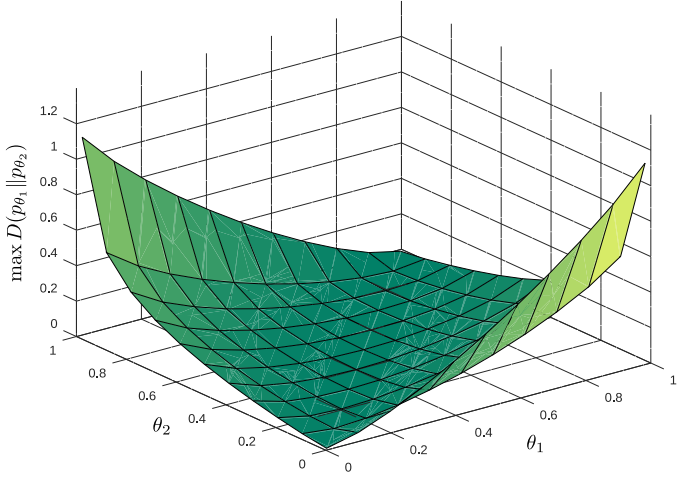}
	\vskip2ex
	\includegraphics[scale=0.48]{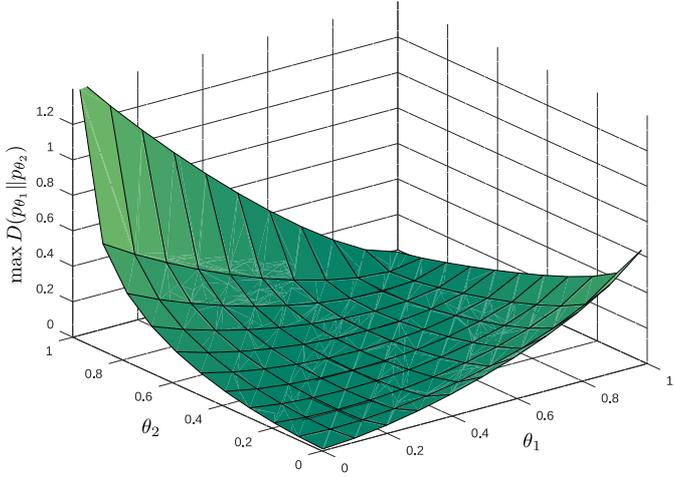}
	\caption{The colored surfaces illustrate the maximum values in Corollary~\ref{Renyi} when $s=0$.
	The upper and lower figures illustrate the cases $(\delta,w)=(1/4,1/2)$ and $(\delta,w)=(1/4,2/5)$, respectively.}
	\label{F4}
\end{figure}

Finally, we address the minimization problem \eqref{eq:Han-Kobayashi}.
Using the pair $(\bar{p}_0,\bar{p}_1)$ given in Corollary~\ref{Renyi},
we have
\begin{align*}
	&\min_{\|(1-w)p_0 - w p_1\|_1\le\delta} \sup_{s>0} \frac{s}{1+s}\{ r-D_{1+s} (p_{\theta_2}\|p_{\theta_1}) \}\\
	\ge& \sup_{s>0} \frac{s}{1+s}\Bigl\{ r - \max_{\|(1-w)p_0 - w p_1\|_1\le\delta} D_{1+s} (p_{\theta_2}\|p_{\theta_1}) \Bigr\}\\
	\overset{(a)}{=}& \sup_{s>0} \frac{s}{1+s}\{ r - D_{1+s}(\bar{p}_{\theta_2}\|\bar{p}_{\theta_1}) \}\\
	\ge& \min_{\|(1-w)p_0 - w p_1\|_1\le\delta} \sup_{s>0} \frac{s}{1+s}\{ r - D_{1+s} (p_{\theta_2}\|p_{\theta_1}) \},
\end{align*}
where $(a)$ follows from Corollary~\ref{Renyi}.
Therefore, the minimum exponent \eqref{eq:Han-Kobayashi} is given by 
\begin{align*}
	&\sup_{s>0} \frac{s}{1+s}\Bigl\{ r - \max_{\|(1-w)p_0 - w p_1\|_1\le\delta} D_{1+s} (p_{\theta_2}\|p_{\theta_1}) \Bigr\}\\
	=& \sup_{s>0} \frac{s}{1+s}\{ r - D_{1+s} (\bar{p}_{\theta_2}\|\bar{p}_{\theta_1}) \}.
\end{align*}
Hence this minimum exponent can be calculated by using Corollary~\ref{Renyi}.

\section{Related work}\label{S6}
Here we compare our result with preceding studies.
The earlier studies \cite{Warner,Greenberg} discussed the estimation error in a similar way, 
but they did not give any privacy criteria.
Warner \cite{Warner} proposed a scheme to protect individual privacy, in 
which each individual flips each true answer with probability $1-\pi$ 
and does not flip it with probability $\pi\in(0,1)$.
That is, he proposed to set $p_0$ and $p_1$ as $p_0 = [\pi,1-\pi]$ and $p_1 = [1-\pi,\pi]$ in our notation.
Greenberg et al.\ \cite{Greenberg} proposed another scheme by using a question unrelated 
to an intended YES/NO question.
In their scheme, the investigator asks each individual both the intended question and the unrelated one.
The asked individual answers the former with probability $\pi\in(0,1)$ and the latter with probability $1-\pi$.
When the unrelated question has the true ratio $\eta:1-\eta$, 
the distributions $p_0$ and $p_1$ are set to
$p_0 = [\pi+(1-\pi)(1-\eta),(1-\pi)\eta]$ and $p_1 = [(1-\pi)(1-\eta),\pi+(1-\pi)\eta]$ in our notation. 
Maximizing the Fisher information in two sets of the above respective pairs $(p_0,p_1)$, 
we obtain the optimal pairs in Table~\ref{T1}.
Our scheme is best of the pairs in Table~\ref{T1} 
because our scheme is to maximize the Fisher information in the set of all pairs of two distributions. 
Moreover, the studies \cite{Warner,Greenberg} did not consider the Fisher information and 
considered only the case $|\mathcal{X}|=|\mathcal{Y}|=2$. 
Hence, even if their schemes are optimized, 
it is impossible to surpass our optimal performance not only the blue broken curves but also the red solid curves in Fig.~\ref{F2}.
This impossibility is illustrated in Fig.~\ref{F3}.

\begin{table}[t]
\caption{Comparison with existing results based on UC-security measure}\label{T1}
This table shows optimal pairs of two distributions $p_0$ and $p_1$ in \cite{Warner}, \cite{Greenberg}, \cite{HLM}, and ours.
\begin{center}
\begin{tabular}{|c||c|c|c|}
	\hline
	Scheme&$|\mathcal{Y}|$&\multicolumn{2}{|c|}{Optimal pair of two distributions}\\
	\hline\hline
	\multirow{2}{*}{Warner \cite{Warner}}&\multirow{2}{*}{$2$}&\multicolumn{2}{|c|}{$p_0 = [(1+\delta)/2,(1-\delta)/2]$}\\
	   &   &\multicolumn{2}{|c|}{$p_1 = [(1-\delta)/2,(1+\delta)/2]$}\\
	\hline
	\multirow{2}{*}{Unrelated question \cite{Greenberg}}&\multirow{2}{*}{$2$}&\multicolumn{2}{|c|}{$p_0 = [\delta+(1-\delta)(1-\eta),(1-\delta)\eta]$}\\
	   &   &\multicolumn{2}{|c|}{$p_1 = [(1-\delta)(1-\eta),\delta+(1-\delta)\eta]$}\\
	\hline
	\multirow{3}{*}{Holohan et al.\ \cite{HLM}}&\multirow{3}{*}{$2$}&$\theta\le1/2$&$\theta>1/2$\\
	\cline{3-4}
	   &   &$p_0 = [1,0]$&$p_0 = [1-\delta,\delta]$\\
	   &   &$p_1 = [1-\delta,\delta]$&$p_1 = [1,0]$\\
	\hline
	\multirow{2}{*}{This paper}&\multirow{2}{*}{$3$}&\multicolumn{2}{|c|}{$p_0 = [2a, 1-2a, 0] = [1-\delta, \delta, 0]$}\\
	   &   &\multicolumn{2}{|c|}{$p_1 = [2a, 0, 1-2a] = [1-\delta, 0, \delta]$}\\
	\hline
\end{tabular}
\end{center}
\end{table}

\begin{figure}[t]
	\centering
	\includegraphics[scale=0.5]{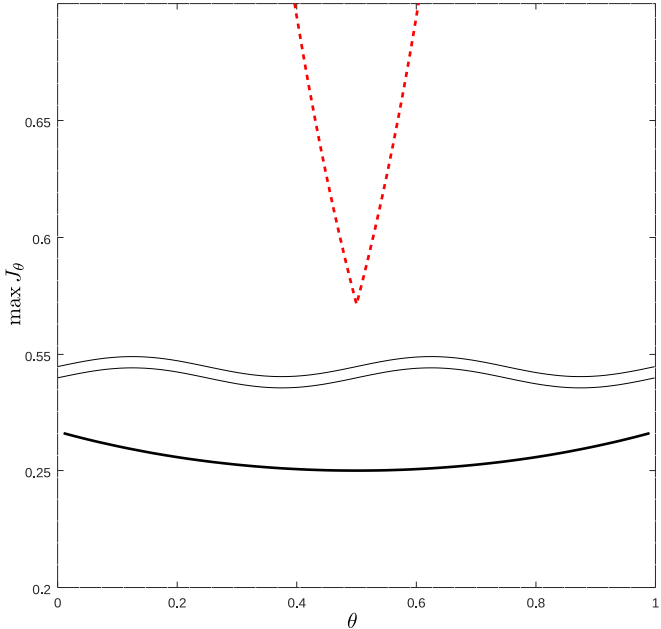}
	\caption{$(\delta,w)=(1/4,1/2)$. 
	The red broken curve illustrates the maximum value in Theorem~\ref{c-var} when $|\mathcal{Y}|=2$, 
	which is also the Fisher information for the third pair in Table~\ref{T1}. 
	The black solid curve illustrates the Fisher information for the first pair in Table~\ref{T1}. 
	The black solid curve also illustrates the Fisher information for the second pair with $\eta=1/2$ in Table~\ref{T1}.}
	\label{F3}
\end{figure}

As a privacy measure, many studies adopt \textit{differential privacy}. 
More precisely, \textit{$(\epsilon,\delta)$-differential privacy} is usually used. 
In our setting, it is defined as \eqref{DP}. 
On the one hand, $(\epsilon,0)$-differential privacy evaluates the maximum of the correct guessing probability 
of a private data $X_i$ when knowing a disclosed data $Y_i$. 
On the other hand, $(0,\delta)$-differential privacy evaluates the average of the correct guessing probability. 
The relation between $(0,\delta)$-differential privacy and the average error probability 
has been already explained in Section~\ref{S1}.

The studies \cite{HLM,KOV16, GV16a} are closely related to this paper. 
Holohan et al.\ \cite{HLM} showed optimal $(\epsilon,\delta)$-differentially private mechanisms explicitly in the case $|\mathcal{X}|=2$. 
However, they considered only the case $|\mathcal{Y}|=2$. 
Hence, when $\epsilon=0$, their result \cite[Theorem~3]{HLM} corresponds to the case (i) in Theorem~\ref{c-var}, 
but the case (iii) in Theorem~\ref{c-var} is not known at all. 
Their result in the case $\epsilon=0$ is in Table~\ref{T1} 
and illustrated by the red broken curves in Fig.~\ref{F2} and \ref{F3}. 
Kairouz et al.\ \cite{KOV16} provided a theorem on optimal $(\epsilon,0)$-differentially private mechanisms in the case $|\mathcal{X}|\ge2$. 
Their theorem can be applied to many objective functions including the Fisher information and the $f$-divergence, and 
turns convex optimization problems to linear programs. 
As stated in the previous sections, their objective functions are the same as ours. 
Also, recall that the contents of Section~\ref{S5} can be regarded as a special case of Theorem~\ref{thm:subl}. 
Hence, in the case $|\mathcal{X}|=2$, Theorem~\ref{thm:subl} can be regarded as the $(0,\delta)$-differential privacy version of the result in \cite{KOV16}. 
Table~\ref{T2} in Section~\ref{S1} summarizes the relation among \cite{KOV16}, \cite{HLM} and this paper. 
Moreover, in the case $\mathcal{X}=\mathbb{Z}$, 
Geng and Viswanath \cite{GV16a} discussed minimization of $l^1$ and $l^2$ cost functions under the $(\epsilon,\delta)$-differential privacy constraint 
and gave lower and upper bounds of the minimum values. 
As a special case, when $\epsilon=0$ and $\delta\to0$, their lower and upper bounds are equal asymptotically. 
Their scheme to protect individual privacy is different from ours because 
their scheme is to add uniform noise or discrete Laplacian noise to integers. 
Indeed, our randomization scheme allows that added noise depends on values of $\mathcal{X}$ 
but their scheme does not. 
In this sense, our randomization scheme is more general than theirs.

There are other related works on optimization under the differential privacy constraint. 
For instance, Duchi et al.\ \cite{DJW} evaluated the infimums of minimax-type objective functions under the $(\epsilon,0)$-differential privacy constraint. 
They did not minimize those objective functions but provided sharp lower and upper bounds up to constant factors. 
The studies \cite{GKOV,GV16b} focused on staircase mechanisms which are mixtures of a finite number of uniform distributions, in the continuous data case. 
Geng et al.\ \cite{GKOV} optimized the expectation of the $l^1$-norm of a disclosed data $Y\in\mathbb{R}^d$ under the $(\epsilon,0)$-differential privacy constraint, and 
showed that a staircase mechanism is optimal when $d=1,2$. 
Similarly, Geng and Viswanath \cite{GV16b} optimized the expectations of symmetric and increasing cost functions of a disclosed data $Y\in\mathbb{R}^d$ under the $(\epsilon,0)$-differential privacy constraint, and 
showed that staircase mechanisms are optimal when $d=1$. 
Also, there is a work which focused on the relation between mutual information and differential privacy \cite{CY}. 

As other privacy measures, 
Issa and Wagner \cite{IW} considered the decreasing exponent rate of the probability that an eavesdropper exactly estimate a source sequence. 
Their privacy measure is asymptotic, while our privacy measure is not so. 
Agrawal and Aggarwal \cite{AA} focused on only the continuous data case and 
adopted another privacy measure, which is given as entropy. 
They defined information loss as the expectation of the $L^1$-metric between the true distribution and its estimate, and 
discussed the trade-off relation between information loss and privacy. 
Also, they proposed an expectation-maximization algorithm as a method to estimate statistical information.

\section{Concluding remarks} \label{S7}
In conclusion, for the trade-off problem associated with binary private data, 
we have proposed a randomized mechanism that can maximize the estimation accuracy of a global property 
while keeping individual privacy a given level. 
Since we assume that private data are used for another cryptographic protocol like authentication, 
the UC-security measure is suitable for a privacy criterion. 
In our setting, the UC-security measure is the $l^1$-norm $(1/2)\|p_0-p_1\|_1$ and 
the constraint $(1/2)\|p_0-p_1\|_1\le\delta$ can be regarded as $(0,\delta)$-differential privacy. 
Under this constraint, we have maximized the Fisher information that is the estimation accuracy of a global property. 
In particular, to achieve the maximum value, 
the set of all randomized data must consist of at least three elements: $|\mathcal{Y}|\ge3$. 
This fact is new and different from the $(\epsilon,0)$-differential privacy case 
because the $(\epsilon,0)$-differential privacy case is maximized even when $|\mathcal{X}|=|\mathcal{Y}|$ \cite{KOV16}. 
We have also extended our analysis to the maximization of the Chernoff bound, which expresses 
the optimal exponentially decreasing rate of discrimination. 
These results may have practical applications in scenarios like voting and surveying.

Our optimal distributions in Theorem~\ref{thm:subl} do not depend on objective functions; 
however, in the case $|\mathcal{X}|\ge3$, optimal mechanisms probably depend on objective functions. 
Hence it is difficult to give optimal mechanisms explicitly in the case $|\mathcal{X}|\ge3$. 
(For example, when the Fisher information $J_\theta$ is an objective function, it depends on the true parameter $\theta$. 
Thus optimal mechanisms are possible to depend on $\theta$.) 
This fact is a main trouble when extending our result to the case $|\mathcal{X}|\ge3$. 
Moreover, to optimize objective functions under the $l^1$-norm constraint, 
we must consider the general case including $|\mathcal{X}|\not=|\mathcal{Y}|$, 
which is also a main trouble.

\balance

\end{document}